\documentclass{amsart}
\usepackage{mathrsfs,amssymb,mathrsfs, multirow,setspace}
\usepackage{amsmath}
\usepackage{listings}
\usepackage{pdfpages}
\usepackage{caption}
\usepackage[a4paper, total={7in, 9in}]{geometry}
\usepackage{enumerate}
\theoremstyle{plain}
\usepackage{graphicx}
\usepackage{mathtools}
\usepackage[utf8]{inputenc}

\newtheorem{theorem}{Theorem}[section]
\newtheorem{lemma}[theorem]{Lemma}

\theoremstyle{definition}

\theoremstyle{remark}

\input xy
\xyoption{all}

\begin{document}
	\title[On rings whose prime ideal sum graphs are line graphs]{On rings whose prime ideal sum graphs are line graphs}

   \author[Praveen Mathil, Jitender Kumar]{Praveen Mathil, Jitender Kumar$^{*}$}
   \address{Department of Mathematics, Birla Institute of Technology and Science Pilani, Pilani-333031, India}
 \email{maithilpraveen@gmail.com,  jitenderarora09@gmail.com}

\begin{abstract}
Let $R$ be a commutative ring with unity. The prime ideal sum graph of the ring $R$ is the simple undirected graph whose vertex set is the set of all nonzero proper ideals of $R$ and two distinct vertices $I$, $J$ are adjacent if and only if $I + J$ is a prime ideal of $R$. In this paper, we characterize all commutative Artinian rings whose prime ideal sum graphs are line graphs. Finally, we give a description of all commutative Artinian rings whose prime ideal sum graph is the complement of a line graph.
\end{abstract}
 \subjclass[2020]{05C25, 13A99}
\keywords{Prime ideal sum graph, Artinian ring, local ring, line graph  \\ *  Corresponding author}
\maketitle

\section{Historical Background and Preliminaries}
The investigation of graphs associated with algebraic structures represents a substantial and vital research area, standing as one of the key focal points within the field of algebraic graph theory. This domain provides a profound interplay between algebra and graph theory, bridging two fundamental branches of mathematics. The comprehensive study of graphs associated with algebraic structures has been of substantial interest due to their applications and their connections to automata theory (see \cite{kelarev2003graph,kelarev2009cayley,kelarev2004labelled}). Various graphs associated with the rings have been studied in the literature, viz: cozero-divisor graph \cite{afkhami2011cozero}, zero divisor graphs \cite{Anderson1999zero}, annihilator graph \cite{badawi2014annihilator}, intersection graph of ideals \cite{chakrabarty2009intersection}, comaximal graph \cite{maimani2008comaximal}, prime ideal sum graph \cite{saha2023prime} etc. 

The prime ideal sum graph of a commutative ring was introduced by Saha \emph{et al.} \cite{saha2023prime}. The \emph{prime ideal sum graph} $\text{PIS}(R)$ of the ring $R$ is the simple undirected graph whose vertex set is the set of all nonzero proper ideals of $R$ and two distinct vertices $I$, $J$ are adjacent if and only if $I + J$ is a prime ideal of $R$. Authors of \cite{saha2023prime} studied some graph-theoretic properties of $\text{PIS}(R)$ such as the clique number, the chromatic number, the domination number etc. The metric dimension and strong metric dimension of the prime ideal sum graph of various classes of rings have been discussed in \cite{adlifard2023metric,mathil2023strong}. Various embeddings of the prime ideal sum graphs on surfaces have been investigated in \cite{mathil2022embedding}. Moreover, they have studied the prime ideal sum graphs with certain forbidden induced subgraphs such as split, threshold, chordal and cograph. The line graph $L(\Gamma)$ of the graph $\Gamma$ is the graph whose vertex set is all the edges of $\Gamma$ and two vertices of $L(\Gamma)$ are adjacent if they are incident in $\Gamma$. Line graphs are described by the nine forbidden subgraphs (cf. Theorem \ref{linegraphchar}). The exploration of algebraic graphs concerning line graphs has been extensively studied by numerous researchers (see \cite{barati2021line,bera2022line,khojasteh2022line,kumar2023finite,pirzada2022line,singh2022graph}). In this paper, we study when the prime ideal sum graph is a line graph. Also, we investigate when the prime ideal sum graph is the complement of a line graph. In Section \ref{section2}, we characterize all commutative Artinian rings whose prime ideal sum graphs are line graphs. In Section \ref{section3}, we characterize all commutative Artinian rings whose prime ideal sum graphs are complements of the line graphs.

We now recall the necessary definitions and results for later use. A \emph{graph} $\Gamma$ is an ordered pair $(V(\Gamma), E(\Gamma))$, where $V(\Gamma)$ is the set of vertices and $E(\Gamma)$ is the set of edges of $\Gamma$. Two distinct vertices $u, v \in V(\Gamma)$ are $\mathit{adjacent}$ in $\Gamma$, denoted by $u \sim v$ (or $(u,  v)$), if there is an edge between $u$ and $v$. Otherwise, we write as $u \nsim v$. Let $\Gamma$ be a graph. A graph $\Gamma'$ is said to be a \emph{subgraph} of $\Gamma$ if $V(\Gamma') \subseteq V(\Gamma)$ and $E(\Gamma') \subseteq E(\Gamma)$. The \emph{complement} $\overline{\Gamma}$ of a graph $\Gamma$ is the graph with vertex set $V(\Gamma)$ and two vertices are adjacent if and only if they are not adjacent in $\Gamma$. If $X \subseteq V(\Gamma)$, then the subgraph $\Gamma(X)$ induced by $X$ is the graph with vertex set $X$ and two vertices of $\Gamma(X)$ are adjacent if and only if they are adjacent in $\Gamma$. A \emph{path} in a graph is a sequence of distinct vertices with the property that each vertex in the sequence is adjacent to the next vertex of it. We use $P_n$ to denote the path graph on $n$ vertices. A graph $\Gamma$ is said to be \emph{complete} if any two vertices are adjacent in $\Gamma$. The complete graph on $n$ vertices is denoted by $K_n$. We denote $mK_n$ by the $m$ copies of $K_n$. A graph $\Gamma$ is said to be \emph{bipartite} if $V(\Gamma)$ can be partitioned into two subsets such that no two vertices in the same partition subset are adjacent. A \emph{complete bipartite} graph is a bipartite graph such that every vertex in one part is adjacent to all the vertices of the other part. A \emph{complete bipartite graph} with partition size $m$ and $n$ is denoted by $K_{m, n}$. 

Throughout the paper, the ring $R$ is a commutative Artinian ring with unity and $F_i$ denotes a field. Further, $U(R)$ denotes the set of all the units of $R$. For basic definitions of ring theory, we refer the reader to  \cite{atiyah1969introduction}. A ring $R$ is said to be \emph{local} if it has a unique maximal ideal $\mathcal{M}$. For a field $F$, we take $\mathcal{M}= \langle 0 \rangle $. By $\mathcal{I}^*(R)$, we mean the set of non-trivial proper ideals of $R$. The \emph{nilpotent index} $\eta(I)$ of an ideal $I$ of the ring $R$ is the smallest positive integer $n$ such that $I^n = 0$. The ideal generated by the elements $x_1, x_2, \ldots, x_n$ ($n \ge 1$) is denoted by $\langle x_1, x_2, \ldots, x_n \rangle$. By the structural theorem \cite{atiyah1969introduction}, an Artinian non-local commutative ring $R$ is uniquely (up to isomorphism) a finite direct product of local rings $R_i$ that is $R \cong R_1 \times R_2 \times \cdots \times R_n$, where $n \geq 2$. 

We will frequently use the following characterizations of the line graphs (see \cite{beineke1970characterizations}).  

 \begin{theorem}\label{linegraphchar} 
 A graph $\Gamma$ is the line graph of some graph if and only if none of the nine graphs in \rm{Figure \ref{forbidden_graphs}} is an induced subgraph of $\Gamma$. 
    \begin{figure}[h!]
			\centering
			\includegraphics[width=1 \textwidth]{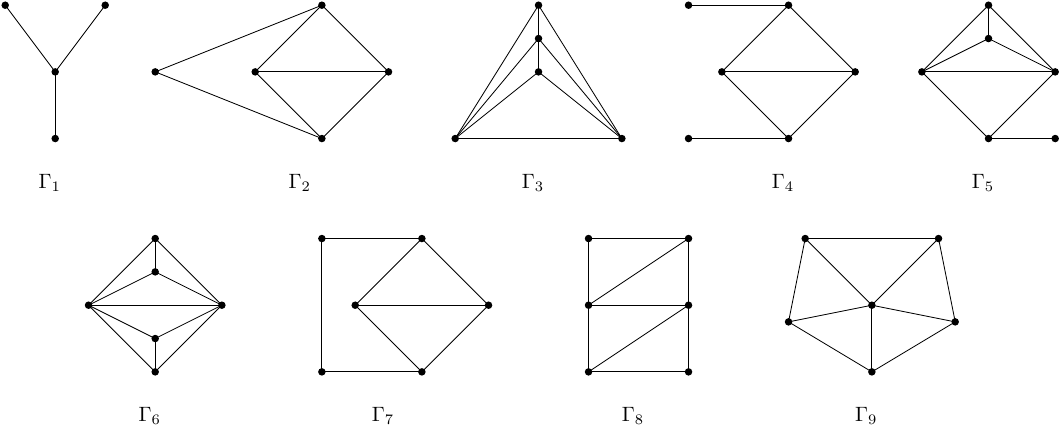}
			\caption{Forbidden induced subgraphs of line graph }
   \label{forbidden_graphs}
\end{figure}
 \end{theorem}

  \begin{theorem}
  A graph $\Gamma$ is the complement of a line graph if and only if none of the nine graphs in \rm{Figure \ref{complenentforbidden_graphs}} is an induced subgraph of $\Gamma$. 
    \begin{figure}[h!]
			\centering
			\includegraphics[width=1 \textwidth]{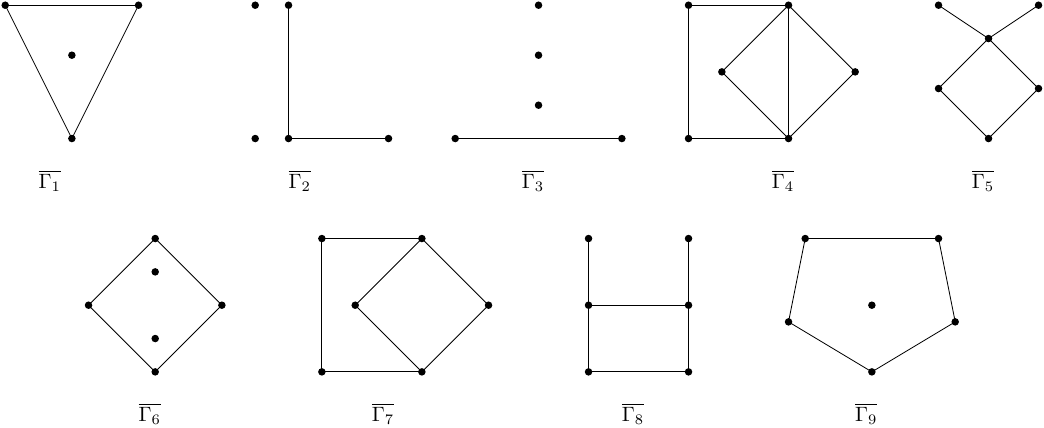}
			\caption{Forbidden induced subgraphs of complement of line graphs }
   \label{complenentforbidden_graphs}
\end{figure}
 \end{theorem}

\section{Prime ideal sum graphs which are line graphs}\label{section2}

In this section, we characterize all the commutative Artinian rings $R$ such that $\textnormal{PIS}(R)$ is a line graph. In the following theorem, first we characterize all the commutative local rings whose prime ideal sum graph is a line graph.

\begin{theorem}\label{primeline}
    Let $R$ be a commutative local ring with maximal ideal $\mathcal{M}$. Then $\textnormal{PIS}(R)$ is a line graph of some graph if and only if one of the following holds.
    \begin{enumerate}[\rm(i)]
        \item $R$ is a principal ideal ring such that $\eta(\mathcal{M}) \le 4$.
        \item $\mathcal{M} = \langle x, y \rangle$ for some $x,y \in R$ such that $x^2 = y^2 =0$. 
    \end{enumerate}
\end{theorem}

\begin{proof}
    Suppose that $\textnormal{PIS}(R)$ is a line graph of some graph. Let $\{ x_1, x_2, \ldots , x_m\}$ be a minimal set of generators of $\mathcal{M}$. First assume that $m=1$, i.e. $R$ is a principal ideal ring. If $\eta(\mathcal{M}) \ge 5$, then the subgraph induced by the set $\{ \mathcal{M}, \mathcal{M}^2, \mathcal{M}^3, \mathcal{M}^4 \}$  is isomorphic to $K_{1,3}$, a contradiction. It implies that $\eta(\mathcal{M}) \le 4$. Now if $m =2$, then $\mathcal{M} = \langle x, y \rangle$ for some $x,y \in R$. For $x^2 \neq 0$, note that the subgraph induced by the set $\{ \mathcal{M}, \langle x \rangle, \langle x^2 \rangle, \langle x^2 +x \rangle  \}$ is isomorphic to $K_{1,3}$, which is not possible. It follows that for $\textnormal{PIS}(R)$ to be a line graph of some graph, we must have $x^2 = y^2 = 0$. If $m \ge 3$, then the subgraph induced by the set $\{ \mathcal{M}, \langle x_1 \rangle, \langle x_2 \rangle, \langle x_3 \rangle  \}$ is isomorphic to $K_{1,3}$, again a contradiction.

    Conversely, first let $R$ be a principal ideal ring such that $\eta(\mathcal{M}) \le 4$. If $\eta(\mathcal{M}) =4$, then note that $\textnormal{PIS}(R) \cong P_3 = L(P_4)$. If $\eta(\mathcal{M}) =3$, then $\textnormal{PIS}(R) \cong K_2 = L(P_3)$. If $\eta(\mathcal{M}) =2$, then $\textnormal{PIS}(R) = L(K_2)$. If $\eta(\mathcal{M}) =1$, i.e. $R$ is a field, then $\textnormal{PIS}(R)$ is an empty graph which is the line graph of the null graph. Next, assume that $\mathcal{M} = \langle x, y \rangle$, for some $x,y \in R$ such that $x^2 = y^2 = 0$. Then all the non-trivial ideals of $R$ are $\mathcal{M}$, $\langle x \rangle$, $\langle y \rangle$, $\langle xy \rangle$, $\langle x + uy \rangle$, where $u \in U(R)$. Now we show that $u-v \in \mathcal{M}$ if and only if $\langle x + uy \rangle = \langle x + vy \rangle$. Let $u-v \in \mathcal{M}$. Then $\langle x + uy \rangle = \langle x + (u-v+v)y \rangle = \langle x + vy \rangle$. Conversely, if $\langle x + uy \rangle = \langle x + vy \rangle$, then $x + uy = \alpha(x + vy)$ for some $\alpha \in R$. It follows that $(1- \alpha)x = (\alpha v -u)y$ and so $(1- \alpha)x \in \langle y \rangle$ and $(\alpha v -u)y \in \langle x \rangle$. Since $\mathcal{M}$ is minimally generated by two elements, we must have $1- \alpha \in \mathcal{M}$ and $\alpha v -u \in \mathcal{M}$. Consequently, $v-u = v(1-\alpha) + \alpha v -u \in \mathcal{M}$. Therefore, we have $|R/\mathcal{M}| + 3$ non-trivial ideals of $R$. 

    Next, we show that for the distinct vertices $\langle x + uy \rangle$ and $ \langle x + vy \rangle$, where $u,v \in U(R)$, we have $\langle x + uy \rangle + \langle x + vy \rangle = \mathcal{M}$. Clearly, $\langle x + uy \rangle + \langle x + vy \rangle \subseteq \mathcal{M}$. Since $v-u \notin \mathcal{M}$, we have $x = v(v-u)^{-1}(x + uy) + [1-v(v-u)^{-1}](x +vy)$. It follows that $x \in \langle x + uy \rangle + \langle x + vy \rangle$. Similarly, $y \in \langle x + uy \rangle + \langle x + vy \rangle$. Consequently, $\mathcal{M} \subseteq \langle x + uy \rangle + \langle x + vy \rangle$. 
    
    Thus, the subgraph induced by the set $\mathcal{I}^*(R) \setminus \{ \langle xy \rangle \}$ is a complete graph. Now, if $xy =0$, then being a complete graph on $|R/\mathcal{M}| + 2$ vertices, $\textnormal{PIS}(R)$ is a line graph of the graph $K_{1, |R/\mathcal{M}| + 2}$. If $xy \neq 0$, then the vertex $\langle xy \rangle$ is of  degree one (adjacent to $\mathcal{M}$ only). Consequently, for $xy \neq 0$, any induced subgraph of $\textnormal{PIS}(R)$ can not be isomorphic to any of the graphs given in Figure \ref{forbidden_graphs}. Thus, $\textnormal{PIS}(R)$ is a line graph of some graph. This completes our proof.    
 \end{proof}

Now we investigate non-local commutative rings whose prime ideal sum graphs are line graphs.

\begin{lemma}\label{fourproduct}
 Let $R \cong R_1 \times R_2 \times \cdots \times R_n$ $(n \ge 2)$ be a non-local commutative ring, where each $R_i$ is a local ring with maximal ideal $\mathcal{M}_i$. If $n \ge 4$, then $\textnormal{PIS}(R)$ is not a line graph.   
\end{lemma}

\begin{proof}
    To prove the result, it is sufficient to show that $\textnormal{PIS}(R)$ is not a line graph when $n=4$. Let $R \cong R_1 \times R_2 \times R_3 \times R_4$. Consider the set 
    \[ S = \{ \mathcal{M}_1 \times R_2 \times R_3 \times R_4, \mathcal{M}_1 \times \mathcal{M}_2 \times R_3 \times R_4, \mathcal{M}_1 \times \mathcal{M}_2 \times \mathcal{M}_3 \times R_4, \mathcal{M}_1 \times \mathcal{M}_2 \times R_3 \times \mathcal{M}_4 \}. \]
    Then the subgraph induced by the set $S$ is isomorphic to $K_{1,3}$. Consequently, for $n \ge 4$, $\textnormal{PIS}(R)$ is not a line graph. 
\end{proof}

\begin{theorem}
    Let $R \cong R_1 \times R_2 \times \cdots \times R_n$ $(n \ge 2)$ be a non-local commutative ring, where each $R_i$ is a local ring with maximal ideal $\mathcal{M}_i$. Then $\textnormal{PIS}(R)$ is a line graph if and only if one of the following holds.
    \begin{enumerate}[\rm(i)]
        \item $R \cong F_1 \times F_2 \times F_3$.
        \item $R \cong F_1 \times F_2$.
        \item $R \cong R_1 \times F_2$ such that $R_1$ is a principal ideal ring with $\eta({\mathcal{M}_1}) = 2$. 
    \end{enumerate}
\end{theorem}

\begin{proof}
Suppose that $\textnormal{PIS}(R)$ is a line graph of some graph. By Lemma \ref{fourproduct}, we have $n \le 3$. First, let $n =3$ i.e., $R \cong R_1 \times R_2 \times R_3$. Assume that one of $R_i$ is not a field. Without loss of generality, assume that $R_1$ is not a field. Consider the set $A = \{ \mathcal{M}_1 \times R_2 \times R_3, \langle 0 \rangle \times R_2 \times \langle 0 \rangle, \mathcal{M}_1 \times \langle 0 \rangle \times \langle 0 \rangle, \langle 0 \rangle \times \langle 0 \rangle \times R_3\}$. Then the subgraph induced by $A$ is isomorphic to $K_{1,3}$, a contradiction. It implies that $R\cong F_1 \times F_2 \times F_3$.

 Now let $R\cong R_1 \times R_2$. Assume that both $R_1$ and $R_2$ are not fields. Consider the set $B = \{  R_1 \times \mathcal{M}_2, \mathcal{M}_1 \times \langle 0 \rangle, \mathcal{M}_1 \times \mathcal{M}_2, \langle 0 \rangle \times \mathcal{M}_2 \}$. Then $\textnormal{PIS}(B)$ is isomorphic to $K_{1,3}$, which is not possible. It follows that one of $R_i$ must be the field. Without loss of generality, assume that $R_2$ is a field, i.e. $R \cong R_1 \times F_2$. Let $\{ x_1, x_2, \ldots, x_m \}$ be minimal set of generators of the maximal ideal $\mathcal{M}_1$. Assume that $m \ge 2$. Consider the set $C = \{ \mathcal{M}_1 \times F_2, \langle 0 \rangle \times F_2, \langle x_1 \rangle \times \langle 0 \rangle, \langle x_2 \rangle \times \langle 0 \rangle \} $. Then the subgraph induced by $C$ is isomorphic to $K_{1,3}$, a contradiction. It implies that $m=1$, i.e. $R_1$ is a principal ideal ring. Moreover, if $\eta(\mathcal{M}_1) \ge 3$, then the subgraph induced by the set $\{ \mathcal{M}_1 \times F_2, \mathcal{M}_1^2 \times F_2, \langle 0 \rangle \times F_2, \mathcal{M}_1^2 \times \langle 0 \rangle\} $ is isomorphic to $K_{1,3}$, which is not possible. Consequently, either $R \cong R_1 \times F_2$ such that $\eta(\mathcal{M}_1) = 2$ or $R \cong F_1 \times F_2$.

 Conversely, one can observe that $\textnormal{PIS}(F_1 \times F_2 ) = L(2K_2)$ and $\textnormal{PIS}(F_1 \times F_2 \times F_3) = L(H_1)$ (see Figure \ref{linefigure_1}). Also, $\textnormal{PIS}(R_1 \times F_2 ) = L(H_2)$ (see Figure \ref{linefigure_2}), where $\eta(\mathcal{M}_1) = 2$.
 \begin{figure}[h!]
			\centering
			\includegraphics[width=1 \textwidth]{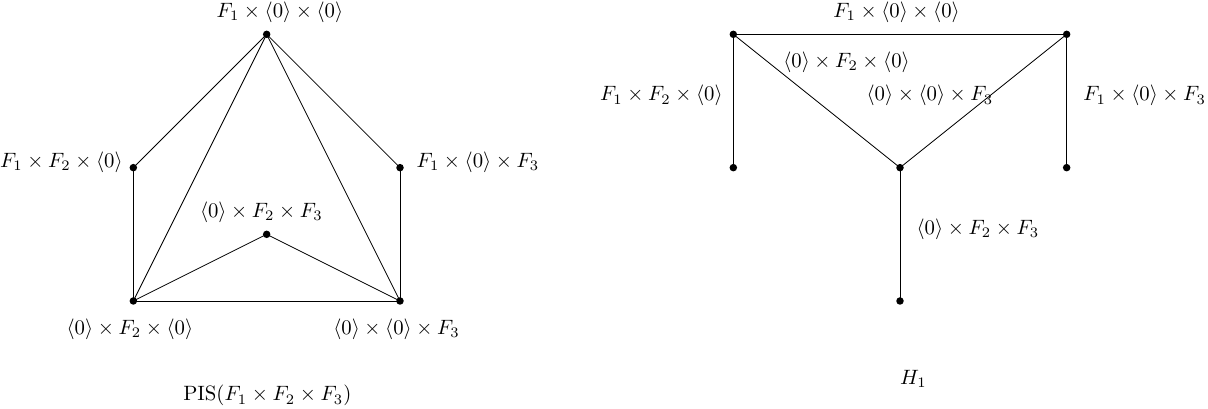}
			\caption{$\textnormal{PIS}(F_1 \times F_2 \times F_3)$ is the line graph of $H_1$ }
   \label{linefigure_1}
\end{figure}
 \begin{figure}[h!]
			\centering
			\includegraphics[width=0.8 \textwidth]{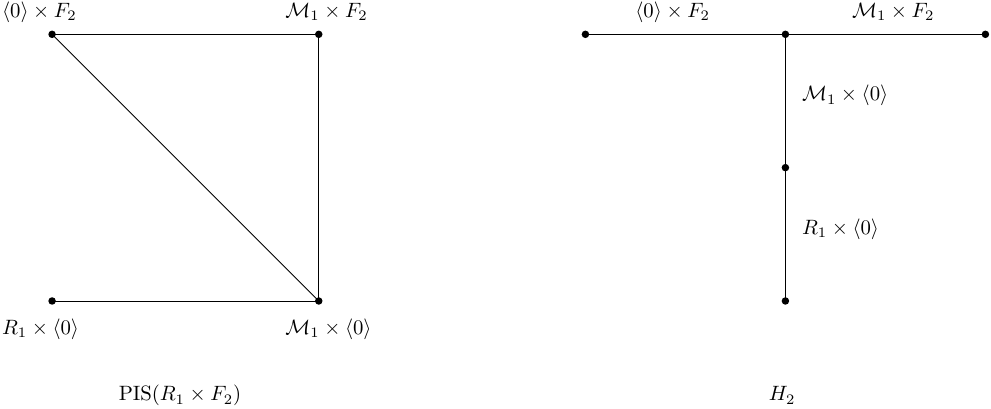}
			\caption{$\textnormal{PIS}(R_1 \times F_2)$ is the line graph of $H_2$ }
   \label{linefigure_2}
\end{figure}
\end{proof}

\section{Prime ideal sum graphs which are the complement of some line graphs}\label{section3}

In this section, we characterize all commutative Artinian rings whose prime ideal sum graphs are the complement of line graphs. We begin to characterize all the commutative local rings $R$ such that $\textnormal{PIS}(R)$ is the complement of a line graph.

\begin{theorem}
    Let $R$ be a commutative local ring with maximal ideal $\mathcal{M}$. Then $\textnormal{PIS}(R)$ is the complement of a line graph if and only if one of the following holds.
    \begin{enumerate}[\rm(i)]
        \item $R$ is a principal ideal ring.
        \item $\mathcal{M} = \langle x, y \rangle$ for some $x,y \in R$ such that $x^2 = y^2 = xy =0$. 
    \end{enumerate}
\end{theorem}

\begin{proof}

Suppose that $\textnormal{PIS}(R)$ is the complement of a line graph. Let $\{ x_1, x_2, \ldots , x_m\}$ be the minimal set of generators of $\mathcal{M}$. If $m \ge 4$, then the subgraph induced by the set $\{ \langle x_4 \rangle, \langle x_2, x_3, \ldots, x_m \rangle, \langle x_1, x_3, \ldots, x_m \rangle, \langle x_1, x_2, x_4, \ldots, x_m \rangle\} $ is isomorphic to $\overline{\Gamma_1}$, a contradiction. For $m =3$, consider the set $A = \{ \langle x_1 \rangle, \langle x_1, x_2 \rangle, \langle x_1, x_3 \rangle, \langle x_2 + x_3\rangle  \}$. Then the subgraph induced by $A$ is isomorphic to $\overline{\Gamma_1}$, which is not possible. Next, let $m=2$, i.e. $\mathcal{M} = \langle x,y \rangle $ for some $x, y \in R$. If $x^2 \neq 0$, then the subgraph induced by the set $\{ \langle x^2 \rangle, \langle x \rangle, \langle y \rangle, \langle x+y \rangle \}$ is isomorphic to $\overline{\Gamma_1}$, which is not possible. If $xy \neq 0$, then the subgraph induced by the set $\{ \langle xy \rangle, \langle x \rangle, \langle y \rangle, \langle x+y \rangle \}$ is isomorphic to $\overline{\Gamma_1}$, again a contradiction. Consequently, we obtain $x^2 = y^2= xy = 0$. Thus, for $\textnormal{PIS}(R)$ to be the complement of a line graph, we must have either $m = 1$, i.e. $R$ is a principal ideal ring or $\mathcal{M} = \langle x,y \rangle $ for some $x, y \in R$ such that $x^2 = y^2= xy = 0$.

   Conversely, first assume that $R$ is a principal ideal ring. Then note that $\textnormal{PIS}(R)$ is a star graph, and so it does not have any induced subgraph which is isomorphic to any one of the graphs given in Figure \ref{complenentforbidden_graphs}. Therefore, $\textnormal{PIS}(R)$ is the complement of a line graph. Next, assume that $\mathcal{M} = \langle x,y \rangle $ for some $x, y \in R$ and $x^2 = y^2= xy = 0$. By the similar argument used in the Theorem \ref{primeline}, we obtain that $\textnormal{PIS}(R)$ is a complete graph on $|R/ \mathcal{M}| +2$ vertices. Further, note that $\textnormal{PIS}(R) = \overline{L\big( (|R/ \mathcal{M}| +2)K_2 \big)}$. Thus, the result holds.
\end{proof}

We Now characterize all non-local commutative rings whose prime ideal sum graphs are the complement of line graphs.

\begin{theorem}
    Let $R \cong R_1 \times R_2 \times \cdots \times R_n$ $(n \ge 2)$ be a non-local commutative ring, where each $R_i$ is a local ring with maximal ideal $\mathcal{M}_i$. Then $\textnormal{PIS}(R)$ is the complement of a line graph if and only if one of the following holds.
    \begin{enumerate}[\rm(i)]
        \item $R \cong F_1 \times F_2 \times F_3$.
        \item $R \cong F_1 \times F_2$.
        \item $R \cong R_1 \times F_2$ such that $R_1$ is a principal ideal ring with $\eta({\mathcal{M}_1}) = 2$. 
    \end{enumerate}
\end{theorem}

\begin{proof}
Suppose that $\textnormal{PIS}(R)$ is the complement of a line graph.  For $n =4$, consider the set 
    \[ S' = \{ \mathcal{M}_1 \times \mathcal{M}_2 \times R_3 \times R_4, R_1 \times \mathcal{M}_2 \times \mathcal{M}_3 \times R_4, \mathcal{M}_1 \times R_2 \times \mathcal{M}_3 \times R_4, \mathcal{M}_1 \times \mathcal{M}_2 \times \mathcal{M}_3 \times \mathcal R_4 \}. \]
    Then the subgraph induced by the set $S'$ is isomorphic to $\overline{\Gamma_1}$, a contradiction. Consequently, for $n \ge 4$, $\textnormal{PIS}(R)$ is not the complement of a line graph. We may now assume that $n \le 3$. Let $R \cong R_1 \times R_2 \times R_3$. Assume that one of $R_i$ is not a field. Without loss of generality, suppose that $R_1$ is not a field. Then the subgraph induced by the set $X = \{ \mathcal{M}_1 \times R_2 \times R_3, \mathcal{M}_1 \times \langle 0 \rangle \times R_3, \langle 0 \rangle \times R_2 \times R_3, R_1 \times R_2 \times \langle 0 \rangle \} $ is isomorphic to $\overline{\Gamma_1}$, a contradiction. It implies that $R\cong F_1 \times F_2 \times F_3$.

 Next, let $R\cong R_1 \times R_2$. Assume that both $R_1$ and $R_2$ are not fields. Consider the set $Y = \{  \mathcal{M}_1 \times R_2, \langle 0 \rangle \times R_2, \mathcal{M}_1 \times  \langle 0 \rangle, R_1 \times \langle 0 \rangle\}$. Then $\textnormal{PIS}(Y)$ is isomorphic to $\overline{\Gamma_1}$, which is not possible. Consequently, one of $R_i$ must be a field. Without loss of generality, assume that $R_2$ is a field, i.e. $R \cong R_1 \times F_2$. Let $\{ x_1, x_2, \ldots, x_m \}$ be a minimal set of generators of the maximal ideal $\mathcal{M}_1$. Assume that $m \ge 2$. Then the subgraph induced by the set $Z = \{ \mathcal{M}_1 \times \langle 0 \rangle, \langle 0 \rangle \times R_2, \langle x_1 \rangle \times \langle 0 \rangle, \langle x_2 \rangle \times \langle 0 \rangle, \langle x_1 + x_2 \rangle \times \langle 0 \rangle \} $ is isomorphic to $\overline{\Gamma_3}$, a contradiction. It implies that $m=1$ and so $R_1$ is a principal ideal ring. If $\eta(\mathcal{M}_1) \ge 3$, then the subgraph induced by the set $\{ \mathcal{M}_1 \times \langle 0 \rangle, \mathcal{M}_1^2 \times F_2, \mathcal{M}_1 \times F_2, \langle 0 \rangle \times F_2, \mathcal{M}_1^2 \times \langle 0 \rangle, R_1 \times \langle 0 \rangle\} $ is isomorphic to $\overline{\Gamma_4}$ (see Figure \ref{linefigure_5}), which is not possible. Thus, either $R \cong R_1 \times F_2$ such that $\eta(\mathcal{M}_1) = 2$ or $R \cong F_1 \times F_2$.
  \begin{figure}[h!]
			\centering
			\includegraphics[width=0.55 \textwidth]{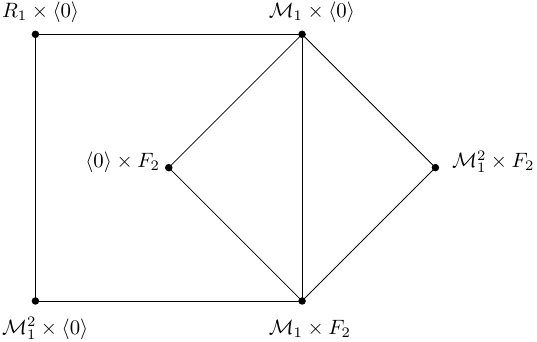}
			\caption{ Induced subgraph of $\textnormal{PIS}(R_1 \times F_2)$, where $\eta(\mathcal{M}_1) \ge 3$. }
   \label{linefigure_5}
\end{figure}

 Conversely, one can observe that $\textnormal{PIS}(F_1 \times F_2 ) = \overline{L(P_3)}$ and $\textnormal{PIS}(F_1 \times F_2 \times F_3) = \overline{L(H_1')}$ (see Figure \ref{linefigure_3}). By Figure \ref{linefigure_4}, we obtain $\textnormal{PIS}(R_1 \times F_2 ) = \overline{L(H_2')}$ with $\eta(\mathcal{M}_1) = 2$.
 \begin{figure}[h!]
			\centering
			\includegraphics[width=1 \textwidth]{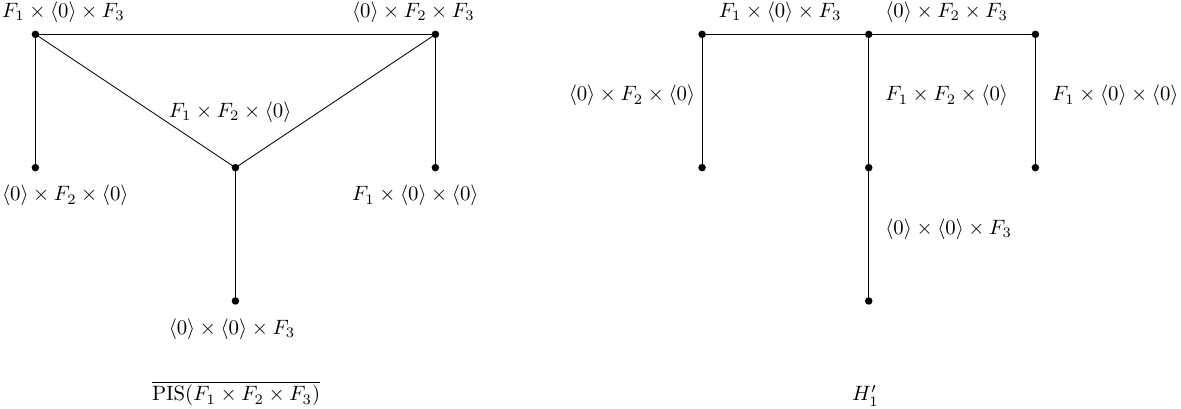}
			\caption{$\overline{\textnormal{PIS}(F_1 \times F_2 \times F_3)}$ is the line graph of $H_1'$ }
   \label{linefigure_3}
\end{figure}
 \begin{figure}[h!]
			\centering
			\includegraphics[width=1 \textwidth]{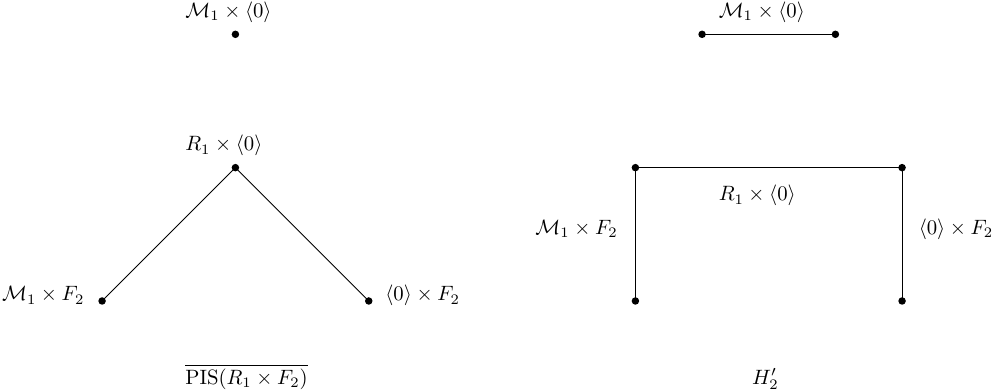}
			\caption{$\overline{\textnormal{PIS}(R_1 \times F_2)}$ is the line graph of $H_2'$ }
   \label{linefigure_4}
\end{figure}
\end{proof}

\newpage
\textbf{Acknowledgement:} The first author gratefully acknowledges Birla Institute of Technology and Science (BITS) Pilani, India, for providing financial support.


\end{document}